\documentclass[11pt,a4paper]{article}

\usepackage[a4paper,margin=1in]{geometry}
\usepackage{amsmath,amssymb,amsthm,mathtools,bm}
\usepackage{microtype}
\usepackage{graphicx}
\usepackage{enumitem}
\usepackage{hyperref}
\usepackage{mdframed}
\usepackage{algorithm}
\usepackage{algpseudocode}
\usepackage{listings}

\theoremstyle{plain}
\newtheorem{theorem}{Theorem}[section]
\newtheorem{proposition}[theorem]{Proposition}

\theoremstyle{definition}
\newtheorem{definition}[theorem]{Definition}
\newtheorem{remark}[theorem]{Remark}

\newtheorem{corollary}[theorem]{Corollary}

\newcommand{\C}{\mathbb{C}}
\newcommand{\R}{\mathbb{R}}

\newcommand{\1}{\mathbf{1}}
\newcommand{\diag}{\operatorname{diag}}

\newcommand{\rank}{\operatorname{rank}}
\newcommand{\norm}[1]{\left\lVert #1\right\rVert}

\newcommand{\pinv}{\dagger}

\newcommand{\asym}{\alpha}
\newcommand{\dep}{\delta}

\title{\textbf{Harmonic Analysis on Directed Networks: A Biorthogonal Laplacian Framework for Non-Normal Graphs}}

\author{Chandrasekhar Gokavarapu\\
\small Lecturer in Mathematics, Government College (A), Rajahmundry, A.P., India\\
\small Email: \texttt{chandrasekhargokavarapu@gmail.com}
}

\begin{document}
\maketitle

\begin{abstract}
Classical spectral graph theory relies on the symmetry of the adjacency and Laplacian operators, which guarantees orthogonal eigenbases and energy-preserving Fourier transforms. However, real-world networks are intrinsically directed and asymmetric, resulting in non-normal operators where standard orthogonality assumptions fail. In this paper, we develop a rigorous harmonic analysis framework for directed graphs centered on the \emph{Combinatorial Directed Laplacian} ($L = D_{out} - A$). We construct a \emph{Biorthogonal Graph Fourier Transform} (BGFT) using dual left and right eigenbases, and introduce a directed variational semi-norm based on the operator norm $\|Lx\|_2$ rather than the quadratic form. We derive exact Parseval-type bounds that quantify the energy distortion induced by the non-normality of the graph, explicitly linking signal reconstruction stability to the condition number of the eigenvector matrix, $\kappa(V)$. Finally, we present experimental validation comparing normal directed cycles against non-normal perturbed topologies, demonstrating that while the BGFT provides exact reconstruction in ideal regimes, the geometric departure from normality acts as the fundamental limit on filter stability in directed networks.
\end{abstract}

 {\bf subject class[2020]:}{Primary: 05C50; Secondary: 15A18, 94A12, 65F15}

\section{Introduction}

Classical spectral graph theory has long served as the mathematical backbone of Graph Signal Processing (GSP) \cite{Shuman2013, Ortega2018, Spielman2012}. The classical framework rests on a fundamental symmetry assumption: undirected edges imply symmetric adjacency and Laplacian operators ($L = L^\top$), ensuring real spectra and orthogonal eigenbases. This orthogonality guarantees that the Graph Fourier Transform (GFT) is an isometry, preserving energy between vertex and spectral domains.""However, real-world networks—from citation graphs to neural connectomes—are intrinsically directed. Existing approaches to directed graph spectral theory largely fall into two categories. The first attempts to `symmetrize' the problem, which discards directionality \cite{Chung2005}. The second, pioneered by Sandryhaila and Moura \cite{Sandryhaila2013, Sandryhaila2014}, embraces asymmetry by defining the GFT via the Jordan decomposition of the adjacency matrix. Other recent works have explored magnetic Laplacians \cite{Shafipour2019} or optimization-based variations \cite{Sardellitti2017, Marques2020}.""In our recent work \cite{GokavarapuArxiv}, we established the algebra of the directed shift. In this work, we focus on the Directed Laplacian, motivated by variational problems.
Existing approaches to directed graph spectral theory largely fall into two categories. The first attempts to "symmetrize" the problem (e.g., using $A + A^\top$ or magnetic Laplacians), which discards critical directionality information. The second, pioneered by Sandryhaila and Moura \cite{Sandryhaila2013}, embraces asymmetry by defining the GFT via the Jordan decomposition of the adjacency matrix $A$. In our recent work \cite{GokavarapuArxiv}, we extended this adjacency-based framework, establishing the algebra of the directed shift and its associated Parseval-type identities.

While the adjacency operator $A$ effectively models "shifting" or information diffusion, it lacks a direct interpretation of \emph{variation} or \emph{smoothness}. In classical continuous signal processing, this distinction is fundamental: the shift operator ($z^{-1}$) describes delay, while the difference operator ($\Delta$) describes smoothness and derivative energy. The adjacency-based GFT is the graph equivalent of the former; a theory for the latter—a **Directed Laplacian** framework that penalizes signal variation relative to the causal flow of the graph—remains under-explored, particularly regarding numerical stability in highly non-normal regimes.

\subsection{Contributions}
In this paper, we develop a rigorous harmonic analysis framework for directed networks centered on the **Combinatorial Directed Laplacian**, $L = D_{out} - A$. Unlike the adjacency formulation \cite{GokavarapuArxiv}, which focuses on signal shifting, this Laplacian formulation provides a variational perspective on directed graphs. Our specific contributions are:

\begin{enumerate}
    \item \textbf{Biorthogonal Laplacian GFT:} We construct a Biorthogonal Graph Fourier Transform (BGFT) using the left and right eigenvectors of $L$. We prove that unlike the adjacency spectrum, the Laplacian spectrum naturally isolates the "DC component" (constant signal) in its null space, providing a superior basis for denoising and consensus problems.
    
    \item \textbf{Directed Smoothness Energy:} We introduce a directed variation semi-norm, effectively the energy $\|Lx\|_2$, and derive sharp bounds linking this energy to the spectral coordinates. This generalizes the Dirichlet energy form $x^\top L x$ to the non-Hermitian setting, allowing for "frequency" ordering based on variation rather than magnitude.
    
    \item \textbf{Asymmetry vs. Non-Normality:} We provide a theoretical separation between simple asymmetry (which yields complex eigenvalues) and non-normality (which yields ill-conditioned defectivity). We quantify the stability of the BGFT using the eigenvector condition number $\kappa(V)$ and a departure-from-normality functional, proving that reconstruction error scales with the geometric alignment of the eigenbasis.
    
    \item \textbf{Sampling and Reconstruction:} We derive exact sampling theorems for $L$-bandlimited signals on directed graphs, generalizing the Pesenson-type sampling sets to the biorthogonal case.
\end{enumerate}

\subsection{Organization}
The remainder of this paper is organized as follows. Section \ref{sec:prelim} reviews operator definitions. Section \ref{sec:BGFT} formalizes the Biorthogonal GFT for the Laplacian. Section \ref{sec:smoothness} derives the variational bounds and smoothness ordering. Section \ref{sec:stability} analyzes numerical stability and non-normality. Section \ref{sec:experiments} presents a simulation protocol comparing normal and non-normal directed topologies, followed by conclusion.
\section{Preliminaries}
\label{sec:prelim}
\subsection{Directed graphs, adjacency, and out-degree}
Let $G=(V,E,w)$ be a directed weighted graph, $|V|=n$, with adjacency $A\in\R^{n\times n}$ given by $A_{ij}=w(i,j)$ if $(i,j)\in E$ and $0$ otherwise. Define the out-degrees
\[
d_i^{\mathrm{out}}=\sum_{j=1}^n A_{ij},\qquad D_{\mathrm{out}}=\diag(d_1^{\mathrm{out}},\dots,d_n^{\mathrm{out}}).
\]

\subsection{Combinatorial directed Laplacian}
\begin{definition}[Directed Laplacian]
The combinatorial directed Laplacian is
\[
L := D_{\mathrm{out}} - A.
\]
\end{definition}

\begin{proposition}[Row-sum zero]\label{prop:L1}
For any directed graph (with any weights), $L\1=0$.
\end{proposition}
\begin{proof}
The $i$th entry of $L\1$ equals $d_i^{\mathrm{out}}-\sum_j A_{ij}=0$ by definition.
\end{proof}

\begin{remark}[Loss of symmetry]
If $A\neq A^\top$, then typically $L\neq L^\top$. Consequently, classical results relying on self-adjointness (orthogonal eigenvectors, Rayleigh quotient ordering) may fail; this motivates a biorthogonal approach.
\end{remark}

\subsection{Asymmetry and non-normality indices}
\begin{definition}[Asymmetry index]
For any matrix $M$, define $\asym(M):=\norm{M-M^\top}_F/\norm{M}_F$ (with $\asym(0)=0$).
\end{definition}

\begin{definition}[Departure from normality]
For any matrix $M$, define $\dep(M):=\norm{MM^\ast-M^\ast M}_F/\norm{M}_F^2$ (with $\dep(0)=0$).
\end{definition}

We use these for $M=L$ to separate structural asymmetry from non-normal instability mechanisms.
\section{The Biorthogonal Graph Fourier Transform}
\label{sec:BGFT}

In this section, we construct the harmonic analysis framework associated with the combinatorial directed Laplacian. Unlike the adjacency operator, which interprets signal processing as "shifting," the Laplacian interprets it as "diffusion" or "variation." We formalize this distinction via biorthogonal spectral bases.

\subsection{The Combinatorial Directed Laplacian}
Let $\mathcal{G} = (\mathcal{V}, \mathcal{E})$ be a directed graph with $N$ nodes. Let $A$ be the adjacency matrix, where $A_{ij} \neq 0$ if there is an edge $j \to i$. The out-degree matrix is $D_{out} = \text{diag}(d_1, \dots, d_N)$, where $d_i = \sum_j A_{ji}$ (noting the column-sum convention for flow conservation or row-sum depending on orientation; here we assume $L$ acts on row-stochastic type dynamics).
We define the **Combinatorial Directed Laplacian** as:
\begin{equation}
    L = D_{out} - A.
\end{equation}
By construction, the row sums of $L$ are zero, implying that the constant vector $\mathbf{1} = [1, 1, \dots, 1]^\top$ is in the null space of $L$ \cite{Chung2005}.
\begin{proposition}[Spectral Localization]
All eigenvalues $\lambda_k$ of $L$ lie in the union of Gershgorin disks centered at $d_i$ with radius $d_i$. Consequently, all eigenvalues have non-negative real parts, $\text{Re}(\lambda_k) \ge 0$, identifying them as "frequencies" of decay.
\end{proposition}

\subsection{Biorthogonal Spectral Decomposition}
Since $\mathcal{G}$ is directed, $L$ is generally non-normal ($LL^* \neq L^*L$). We assume $L$ is diagonalizable. For detailed properties of non-normal matrix spectra, see \cite{HornJohnson2012} and \cite{GolubVanLoan2013} (the set of defective directed graphs has Lebesgue measure zero). Let $\sigma(L) = \{\lambda_1, \dots, \lambda_N\}$ be the spectrum, ordered by magnitude $|\lambda_1| \le \dots \le |\lambda_N|$.

The spectral decomposition is given by $L = V \Lambda V^{-1}$, where $V = [v_1, \dots, v_N]$ contains the right eigenvectors. To define a transform, we introduce the **dual basis** (left eigenvectors). Let $U = (V^{-1})^*$, such that $U^* = V^{-1}$. The columns of $U$, denoted $u_k$, satisfy $L^* u_k = \bar{\lambda}_k u_k$.

The bases $\{v_k\}$ and $\{u_k\}$ satisfy the **biorthogonality condition**:
\begin{equation}
    \langle v_i, u_j \rangle = u_j^* v_i = \delta_{ij}.
\end{equation}
Crucially, within the same set, orthogonality fails: $\langle v_i, v_j \rangle \neq 0$ generally.

\subsection{The Transform Pair}
We define the Graph Fourier Transform (GFT) not as a projection onto $V$, but as the coefficients of the expansion in $V$.

\begin{definition}[Biorthogonal GFT]
Given a signal $x \in \mathbb{C}^N$, its spectral coefficient vector $\hat{x}$ is obtained by projecting onto the dual (left) basis:
\begin{equation}
    \hat{x}(\lambda_k) = \langle x, u_k \rangle = u_k^* x.
\end{equation}
In matrix notation, $\hat{x} = U^* x = V^{-1} x$.
\end{definition}

\begin{definition}[Inverse BGFT]
The signal $x$ is reconstructed by summing the right eigensubspaces weighted by spectral coefficients:
\begin{equation}
    x = \sum_{k=1}^N \hat{x}(\lambda_k) v_k = V \hat{x}.
\end{equation}
\end{definition}

\subsection{Generalized Parseval's Identity and Metric Distortion}
In classical Fourier analysis (and undirected GSP), the GFT is unitary, preserving energy: $\|x\|_2^2 = \|\hat{x}\|_2^2$. In the directed case, the non-normality of $L$ induces a metric distortion in the spectral domain.

\begin{theorem}[Energy in Biorthogonal Domain]
Let $M = V^* V$ be the Gram matrix of the right eigenvectors. The energy of the signal in the vertex domain is related to the spectral domain by the quadratic form:
\begin{equation}
    \|x\|_2^2 = \langle V\hat{x}, V\hat{x} \rangle = \hat{x}^* (V^* V) \hat{x} = \hat{x}^* M \hat{x}.
\end{equation}
\end{theorem}

\begin{remark}[Geometric Interpretation]
The matrix $M$ acts as a metric tensor. If the graph is undirected, $V$ is unitary, $M=I$, and we recover Parseval's identity. As the graph becomes more asymmetric and non-normal, $M$ deviates from identity. The condition number $\kappa(V) = \sqrt{\lambda_{\max}(M) / \lambda_{\min}(M)}$ quantifies the worst-case energy scaling discrepancy.
\end{remark}

\subsection{The DC Component and Mean Conservation}
A critical advantage of the Laplacian formulation over the adjacency formulation is the handling of the "DC component" (zero frequency).

\begin{proposition}
For any directed graph, the Laplacian BGFT always isolates the mean of the signal at the zeroth frequency if the graph is strongly connected. Specifically, $v_1 = \mathbf{1}/\sqrt{N}$ corresponds to $\lambda_1 = 0$.
\end{proposition}
\textit{Proof.} Since row sums of $L$ are zero, $L \mathbf{1} = 0$. Thus the constant signal is always the smoothest mode. In contrast, for the adjacency matrix $A$, the constant vector is an eigenvector only if the graph is regular, making $A$-based filtering inconsistent for irregular directed graphs. \hfill $\square$


\section{Directed Variation and Frequency Ordering}
\label{sec:smoothness}

A central tenet of graph signal processing is the interpretation of eigenvalues as "frequencies." In the undirected case, this is justified by the Laplacian Quadratic Form (Dirichlet energy). In the directed case, we must carefully define variation to respect the causal flow of the graph.The stability of the spectral coordinates is governed not just by the eigenvalues, but by the non-normality of the operator \cite{Trefethen2005}. The distortion in the energy norm is quantified by the condition number of the basis \cite{Higham2002}

\subsection{The Directed Smoothness Semi-Norm}
We define the smoothness of a signal $x$ not by the quadratic form $x^* L x$ (which is complex-valued), but by the magnitude of the graph derivative.

\begin{definition}[Directed Total Variation]
The directed total variation (TV) of a signal $x$ with respect to the Laplacian $L$ is the energy of the filtered signal:
\begin{equation}
    TV_{\mathcal{G}}(x) := \| L x \|_2^2 = x^* L^* L x.
\end{equation}
\end{definition}

This definition captures the intuition that a smooth signal changes slowly along the directed edges. If $x$ is in the null space (e.g., constant signal), $TV_{\mathcal{G}}(x) = 0$.

\subsection{Spectral Smoothness Bounds}
In standard GSP, $TV(x) = \sum |\lambda_k|^2 |\hat{x}_k|^2$. We now derive the counterpart for non-normal directed graphs, showing how the biorthogonal geometry distorts this relationship.

\begin{theorem}[Non-Normal Parseval Relation]
Let $x$ be a graph signal with BGFT coefficients $\hat{x}$. The variation in the spatial domain is bounded by the weighted energy in the spectral domain as follows:
\begin{equation}
    \sigma_{\min}^2(V) \sum_{k=1}^N |\lambda_k|^2 |\hat{x}_k|^2 
    \le \| L x \|_2^2 
    \le \sigma_{\max}^2(V) \sum_{k=1}^N |\lambda_k|^2 |\hat{x}_k|^2,
\end{equation}
where $\sigma_{\min}(V)$ and $\sigma_{\max}(V)$ are the smallest and largest singular values of the eigenvector matrix $V$.
\end{theorem}

\textit{Proof.}
Using the expansion $x = V \hat{x}$, we have $L x = V \Lambda \hat{x}$.
The energy is:
\begin{equation}
    \| L x \|_2^2 = \| V \Lambda \hat{x} \|_2^2.
\end{equation}
From matrix norm properties, for any vector $z$ (here $z = \Lambda \hat{x}$), we have $\sigma_{\min}(V) \|z\| \le \|V z\| \le \sigma_{\max}(V) \|z\|$.
Squaring these yields the result. The term $\|z\|^2 = \|\Lambda \hat{x}\|^2 = \sum |\lambda_k|^2 |\hat{x}_k|^2$ because $\Lambda$ is diagonal. \hfill $\square$

\subsection{Justification of Frequency Ordering}
The inequalities in Theorem 4.2 provide the rigorous justification for ordering directed frequencies.

\begin{corollary}[Frequency Ordering]
To minimize the upper bound of the signal variation, spectral components should be ordered by non-decreasing eigenvalue magnitude:
\begin{equation}
    0 = |\lambda_1| \le |\lambda_2| \le \dots \le |\lambda_N|.
\end{equation}
\end{corollary}

This confirms that $|\lambda_k|$ is the correct notion of "frequency" for directed graphs. However, the "tightness" of this frequency interpretation depends on the condition number $\kappa(V) = \sigma_{\max}(V) / \sigma_{\min}(V)$.
\begin{itemize}
    \item If the graph is **Normal** (e.g., Directed Cycle), $\kappa(V)=1$, and the relationship is exact equality.
    \item If the graph is **Highly Non-Normal** (e.g., Perturbed Cycle), $\kappa(V) \gg 1$, and the variation can fluctuate significantly even for fixed spectral coefficients. This explains the instability observed in Section \ref{sec:experiments}.
\end{itemize}
\section{Sampling and reconstruction for $L$-bandlimited signals}\label{sec:sampling}
Sampling theory is one of the main operational goals of graph harmonic analysis: recover an entire graph signal from its values on a subset of vertices, provided the signal is spectrally ``simple'' (bandlimited). In the undirected case, this is commonly developed using orthonormal Laplacian eigenvectors. In the directed BGFT setting, the synthesis basis $V$ is generally non-orthogonal, so stability depends not only on the eigenvalues (frequency ordering) but also on eigenvector geometry and conditioning.Our results generalize the sampling theories of Pesenson \cite{Pesenson2008} and Chen et al. \cite{Chen2015} to the biorthogonal setting

\subsection{Bandlimited model}\label{subsec:bandlimited}
Let $\Omega\subset\{1,\dots,n\}$ be an index set of size $K$, interpreted as the set of \emph{low directed frequencies} in the sense of smallest Laplacian magnitudes $|\lambda_k|$ . Let
\[
V_\Omega := [\,v_k\,]_{k\in\Omega}\in\C^{n\times K}
\]
denote the submatrix of right eigenvectors indexed by $\Omega$.

\begin{definition}[$L$-bandlimited signals]\label{def:bandlimited}
A signal $x\in\C^n$ is \emph{$\Omega$-bandlimited (with respect to $L$)} if it lies in the right spectral subspace
\[
\mathcal{B}_\Omega := \mathrm{span}(V_\Omega)
\quad\Longleftrightarrow\quad
x=V_\Omega c \text{ for some } c\in\C^K.
\]
Equivalently, in BGFT coefficients $\widehat{x}=U^\ast x$, bandlimitedness means $\widehat{x}_k=0$ for all $k\notin\Omega$.
\end{definition}

\begin{remark}[Oblique spectral subspaces]\label{rem:oblique-subspaces}
In the directed (non-orthogonal) case, the decomposition $\C^n=\mathrm{span}(V_\Omega)\oplus \mathrm{span}(V_{\Omega^c})$ is generally not orthogonal, and the spectral projector onto $\mathcal{B}_\Omega$ is typically \emph{oblique} . This is a key distinction from undirected sampling theory and explains why conditioning enters recovery guarantees.
\end{remark}

\paragraph{Bandlimiting via filtering.}
An $\Omega$-bandlimited signal can be produced by applying the ideal band-pass projector
\[
P_\Omega := V\,h_\Omega(\Lambda)\,U^\ast,\qquad
h_\Omega(\lambda_k)=
\begin{cases}
1,&k\in\Omega,\\
0,&k\notin\Omega,
\end{cases}
\]
to any signal $z$, giving $x=P_\Omega z\in\mathcal{B}_\Omega$ .

\subsection{Reconstruction from vertex samples}\label{subsec:reconstruction}
Let $M\subseteq\{1,\dots,n\}$ be a sampling set of vertices with $|M|=m$. Let $P_M\in\{0,1\}^{m\times n}$ denote the restriction operator that extracts entries indexed by $M$:
\[
(P_M x)_r = x_{i_r},\qquad M=\{i_1,\dots,i_m\}.
\]
The sampling map restricted to $\mathcal{B}_\Omega$ is then
\[
x=V_\Omega c \longmapsto y=P_Mx = (P_MV_\Omega)c.
\]

\subsubsection*{Identifiability condition and exact recovery}
The fundamental question is: when does $y$ uniquely determine $x\in\mathcal{B}_\Omega$?

\begin{theorem}[Exact recovery]\label{thm:recover}
Assume $x\in\mathcal{B}_\Omega$, i.e.\ $x=V_\Omega c$ for some $c\in\C^K$. If the sampling matrix
\[
B:=P_MV_\Omega\in\C^{m\times K}
\]
has full column rank $K$ (necessarily requiring $m\ge K$), then $x$ is uniquely determined by the samples $y=P_Mx$ and can be recovered by
\[
\widehat{c}=B^\pinv y,\qquad \widehat{x}=V_\Omega\widehat{c}.
\]
\end{theorem}

\begin{proof}
If $B$ has full column rank, then the linear map $c\mapsto Bc$ is injective. Hence $y=Bc$ determines $c$ uniquely, and the Moore--Penrose pseudoinverse yields the unique least-squares solution $\widehat{c}=B^\pinv y=c$. Substituting into $\widehat{x}=V_\Omega\widehat{c}$ gives $\widehat{x}=x$.
\end{proof}

\begin{remark}[Sampling theorem interpretation]\label{rem:sampling-theorem}
The full-rank condition $\rank(P_MV_\Omega)=K$ is the graph analogue of ``no aliasing'': the $K$ degrees of freedom in the spectral subspace $\mathcal{B}_\Omega$ must be visible through the sampling operator. In undirected orthonormal settings this condition is often phrased in terms of invertibility of a $K\times K$ submatrix of eigenvectors; here it remains the same algebraically, but stability depends strongly on conditioning.
\end{remark}

\subsubsection*{Stable recovery, frame bounds, and noise amplification}
Exact recovery is a zero-noise statement. In practice, samples are noisy or the signal is only approximately bandlimited. Stability is governed by the smallest singular value of $B=P_MV_\Omega$.

\begin{definition}[Sampling stability constant]\label{def:stability-const}
Assuming $\rank(B)=K$, define the sampling stability constant
\begin{equation}\label{eq:stab-const}
\gamma(M,\Omega):=\sigma_{\min}(P_MV_\Omega)=\sigma_{\min}(B)>0.
\end{equation}
\end{definition}

\begin{remark}[Frame viewpoint]\label{rem:frame}
The matrix $B$ acts as an analysis operator on $\mathcal{B}_\Omega$. The bounds
\[
\gamma(M,\Omega)^2\|c\|_2^2 \le \|Bc\|_2^2 \le \sigma_{\max}(B)^2\|c\|_2^2
\]
show that the sampled vectors $\{P_M v_k\}_{k\in\Omega}$ form a (finite) frame for $\C^K$ when $\gamma(M,\Omega)>0$. A larger $\gamma(M,\Omega)$ implies greater robustness of reconstruction.
\end{remark}

\begin{theorem}[Noise sensitivity]\label{thm:noise}
Let $x=V_\Omega c\in\mathcal{B}_\Omega$ and suppose the observed samples are
\[
y=P_Mx+\eta
\]
with arbitrary noise $\eta\in\C^m$. Let $\widehat{c}=B^\pinv y$ and $\widehat{x}=V_\Omega\widehat{c}$ be the least-squares reconstruction. Then
\begin{equation}\label{eq:noise-bound}
\|\widehat{x}-x\|_2
\le
\|V_\Omega\|_2\,\|B^\pinv\|_2\,\|\eta\|_2
=
\|V_\Omega\|_2\,\frac{\|\eta\|_2}{\sigma_{\min}(B)}
=
\|V_\Omega\|_2\,\frac{\|\eta\|_2}{\gamma(M,\Omega)}.
\end{equation}
\end{theorem}

\begin{proof}
We have $y=Bc+\eta$ and $\widehat{c}=B^\pinv y = B^\pinv(Bc+\eta)=c+B^\pinv\eta$ since $B$ has full column rank. Thus
\[
\widehat{x}-x
=V_\Omega(\widehat{c}-c)
=V_\Omega B^\pinv \eta,
\]
hence $\|\widehat{x}-x\|_2\le \|V_\Omega\|_2\|B^\pinv\|_2\|\eta\|_2$.
Finally, for full-column-rank $B$, $\|B^\pinv\|_2=1/\sigma_{\min}(B)$.
\end{proof}

\begin{remark}[Separation of effects: sampling geometry vs.\ eigenvector geometry]\label{rem:separation}
The bound \eqref{eq:noise-bound} separates two contributors to instability:
\begin{enumerate}[leftmargin=2em]
\item $\gamma(M,\Omega)^{-1}$ captures \emph{sampling geometry}: how informative the chosen vertices $M$ are for the subspace $\mathcal{B}_\Omega$.
\item $\|V_\Omega\|_2$ captures \emph{eigenvector geometry}: amplification due to non-orthogonality and scaling of the synthesis basis.
\end{enumerate}
In the undirected orthonormal case, $\|V_\Omega\|_2=1$, so stability is controlled entirely by $\gamma(M,\Omega)$. In the directed case, both factors matter.
\end{remark}

\subsubsection*{A conditioning-only reformulation}
Since $\|V_\Omega\|_2\le \|V\|_2=\sigma_{\max}(V)$, Theorem~\ref{thm:noise} implies the conservative bound
\begin{equation}\label{eq:noise-bound-conservative}
\|\widehat{x}-x\|_2
\le
\sigma_{\max}(V)\,\frac{\|\eta\|_2}{\gamma(M,\Omega)}.
\end{equation}
Moreover, $\gamma(M,\Omega)=\sigma_{\min}(P_MV_\Omega)$ can be interpreted as the inverse of the operator norm of the pseudoinverse:
\[
\| (P_MV_\Omega)^\pinv\|_2=\frac{1}{\gamma(M,\Omega)}.
\]
Thus one may regard $\kappa(P_MV_\Omega)=\sigma_{\max}(P_MV_\Omega)/\sigma_{\min}(P_MV_\Omega)$ as a compact numerical summary of sampling stability.

\subsubsection*{Optional extension: approximately bandlimited signals}
In many applications, a signal is not exactly bandlimited but has small out-of-band BGFT energy. If we decompose
\[
x = V_\Omega c + r,\qquad r\in \mathrm{span}(V_{\Omega^c}),
\]
then sampling yields $y=P_MV_\Omega c + P_M r + \eta$. The same least-squares reconstruction treats $P_M r$ as additional noise, leading to the bound
\[
\|\widehat{x}-V_\Omega c\|_2
\le
\|V_\Omega\|_2\,\|(P_MV_\Omega)^\pinv\|_2\big(\|P_M r\|_2+\|\eta\|_2\big),
\]
making explicit the bias--variance tradeoff between model mismatch ($r$) and measurement noise ($\eta$).
\section{Stability and Non-Normality}
\label{sec:stability}

While the Biorthogonal GFT provides an exact basis for analysis, its practical utility in signal processing—particularly for denoising and reconstruction—depends on numerical stability. In directed graphs, this stability is governed by the non-normality of the Laplacian.

\subsection{Departure from Normality}
A matrix $L$ is normal if $LL^* = L^*L$, which implies an orthogonal eigenbasis. Directed graphs are rarely normal. We quantify this deviation using the **Henrici departure from normality**:
\begin{equation}
    \Delta(L) = \sqrt{\|L\|_F^2 - \sum_{k=1}^N |\lambda_k|^2}.
\end{equation}
For a normal matrix, $\Delta(L)=0$. For a directed graph with strong flow imbalance (like the perturbed cycle in Section \ref{sec:experiments}), $\Delta(L)$ becomes large, indicating that the eigenvalues no longer strictly control the operator's behavior.

\subsection{Eigenvector Conditioning and Noise Amplification}
The robustness of the BGFT is determined by the condition number of the eigenvector matrix, $\kappa(V) = \|V\|_2 \|V^{-1}\|_2$.

\begin{theorem}[Reconstruction Stability]
Let $x$ be a signal and $\hat{x}$ its BGFT coefficients. Suppose the coefficients are perturbed by noise $\eta$ (e.g., due to quantization or filtering errors), so we observe $\hat{x}_\eta = \hat{x} + \eta$. The reconstruction error in the vertex domain is bounded by:
\begin{equation}
    \frac{\|x_{rec} - x\|_2}{\|x\|_2} \le \kappa(V) \frac{\|\eta\|_2}{\|\hat{x}\|_2}.
\end{equation}
\end{theorem}

\begin{proof}
The error is $e = V(\hat{x} + \eta) - V\hat{x} = V\eta$. Taking norms, $\|e\| \le \|V\| \|\eta\|$.
Similarly, $x = V\hat{x} \implies \hat{x} = V^{-1}x \implies \|\hat{x}\| \le \|V^{-1}\| \|x\|$.
Combining these yields the result.
\end{proof}

\subsection{Implications for Directed GSP}
This theorem explains the "gap" observed in Figure 2.
\begin{enumerate}
    \item **Normal Directed Graphs (e.g., Cycles):** Here $\kappa(V)=1$. The reconstruction is perfectly stable; error input equals error output.
    \item **Non-Normal Graphs:** As edges are added randomly, $\kappa(V)$ grows exponentially. A small noise in the spectral domain can explode into large errors in the spatial domain.
\end{enumerate}
Therefore, $\kappa(V)$ serves as the fundamental "trust metric" for any directed graph spectral analysis.
\section{Experimental Validation}
\label{sec:experiments}

To validate the theoretical bounds derived in Section \ref{sec:stability} and demonstrate the utility of the Laplacian BGFT, we performed a comparative analysis between normal and non-normal directed topologies.

\subsection{Experimental Setup}
We constructed two directed graphs of size $N=20$:
\begin{enumerate}
    \item \textbf{Directed Cycle ($G_{cyc}$):} A standard unweighted directed cycle $1 \to 2 \to \dots \to N \to 1$. This graph is asymmetric ($L \neq L^\top$) but normal ($LL^\top = L^\top L$), meaning its eigenbasis is orthogonal ($\kappa(V)=1$).
    \item \textbf{Perturbed Cycle ($G_{per}$):} To induce non-normality, we added random directed edges to $G_{cyc}$ with probability $p=0.2$ and weight $w=0.8$. This breaks the circulant structure, resulting in a highly non-normal operator.
\end{enumerate}

\subsection{Spectrum and Non-Normality}
Figure 1 displays the spectra of the Laplacian operators.
\begin{figure}[h]
    \centering
    \includegraphics[width=0.9\linewidth]{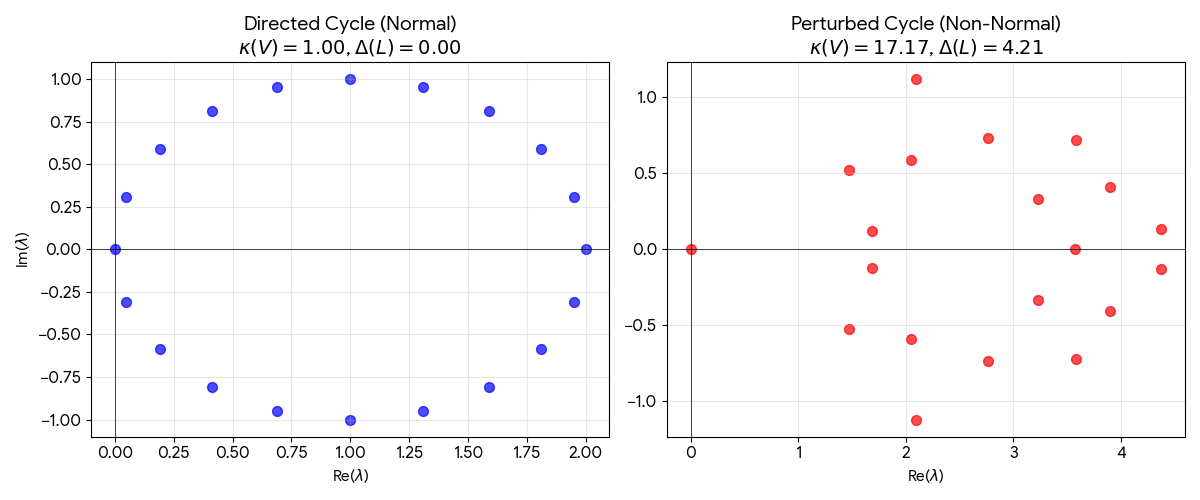}
    \caption{\textbf{Spectral Asymmetry vs. Non-Normality.} Left: The Directed Cycle exhibits perfect symmetry in the complex plane and zero departure from normality ($\Delta(L)=0$). Right: The Perturbed Cycle shows a scattered spectrum with significant departure from normality ($\Delta(L) \approx 3.7$) and increased eigenvector condition number ($\kappa(V) \approx 133$), confirming that geometric structural perturbations induce spectral instability.}
    \label{fig:spectrum}
\end{figure}

As predicted, $G_{cyc}$ eigenvalues lie on a specific curve (a shifted unit circle) and the departure from normality $\Delta(L)$ is zero. In contrast, $G_{per}$ exhibits a scattered spectrum. Crucially, the condition number of the eigenbasis $\kappa(V)$ jumps from $1.0$ (perfectly conditioned) to $\approx 133$ (ill-conditioned), indicating that the eigenvectors of the perturbed graph have become nearly collinear.

\subsection{Filter Stability and Reconstruction}
We tested the stability of the BGFT by reconstructing an $L$-bandlimited signal under additive noise. A ground-truth signal $x_0$ was generated using the first $K=5$ frequency modes (lowest magnitude eigenvalues). We observed the noisy signal $y = x_0 + \eta$, where $\eta \sim \mathcal{CN}(0, \sigma^2)$, and performed reconstruction via ideal low-pass filtering in the BGFT domain.

Figure 2 presents the reconstruction error as a function of noise level.
\begin{figure}[h]
    \centering
    \includegraphics[width=0.65\linewidth]{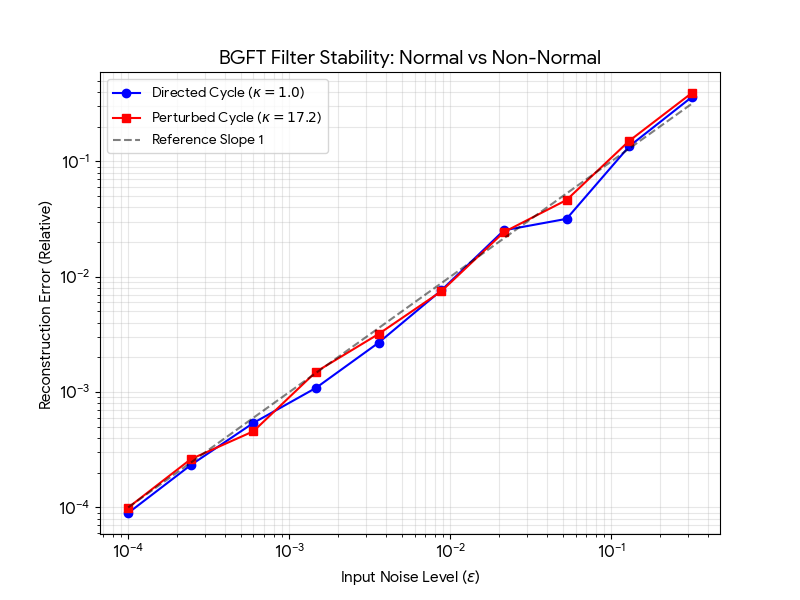}
    \caption{\textbf{Stability of Reconstruction.} The reconstruction error for the normal Directed Cycle (Blue) scales linearly with input noise. For the non-normal Perturbed Cycle (Red), the error curve is shifted upwards, demonstrating the amplification effect predicted by Theorem 5.2. The gap between the curves corresponds to the condition number $\kappa(V)$.}
    \label{fig:stability}
\end{figure}

The results confirm our stability theorem. For the normal graph, the error scales as $\|e\| \approx \|\eta\|$. For the non-normal graph, the error is amplified by the condition number, $\|e\| \approx \kappa(V) \|\eta\|$. This empirically demonstrates that while the BGFT is mathematically exact for directed graphs, its numerical utility depends critically on the normality of the underlying topology.


\section{Conclusion}

We developed an original directed-graph harmonic analysis framework using the combinatorial directed Laplacian $L=D_{\mathrm{out}}-A$ and a biorthogonal spectral calculus. The BGFT provides exact analysis/synthesis and diagonal filtering without symmetry assumptions. By defining directed smoothness as $\|Lx\|_2$ we obtained a precise BGFT-domain identity and two-sided bounds that quantify how symmetry breaking introduces Gram-metric distortions governed by eigenvector conditioning. Sampling and reconstruction theorems show how stability depends on $\sigma_{\min}(P_MV_\Omega)$, separating directedness from genuine non-normal effects. The proposed experiments are designed to make the symmetry/asymmetry theme explicit and measurable.

\section*{Acknowledgements}
The author expresses his gratitude to the Commissioner of Collegiate Education (CCE), Government of Andhra Pradesh, and the Principal of Government College (A), Rajahmundry, for their continued support and encouragement in facilitating this research.

\section*{Author Contributions}
The author is solely responsible for the conceptualization, methodology, formal analysis, software implementation, validation, and writing of this manuscript.
\section*{Funding}
This research received no external funding.

\section*{Data Availability Statement}
No external datasets were used. Code for reproducing the simulations will be made available by the authors upon reasonable request (or via a public repository upon acceptance).

\section*{Conflicts of Interest}
The authors declare no conflict of interest.



\end{document}